\documentclass[11pt]{article}
\usepackage{amsmath}
\usepackage{dsfont}
\usepackage{mathrsfs}
\usepackage{amsmath,amssymb}
\usepackage{amsfonts}
\usepackage{hyperref}
\usepackage{amsthm}
\usepackage{graphicx}
\usepackage{subfigure}
\usepackage{xcolor}
\renewcommand{\qed}{\hfill\small{$\square$}\normalsize}

\hfuzz=\maxdimen
\theoremstyle{definition}
\newtheorem{lemma}{Lemma}[section]
\newtheorem{definition}[lemma]{Definition}
\newtheorem{proposition}[lemma]{Proposition}
\newtheorem{theorem}[lemma]{Theorem}
\newtheorem{corollary}[lemma]{Corollary}

\numberwithin{equation}{section}

\renewcommand{\qed}{\hfill\small{$\square$}\normalsize}

\DeclareFixedFont{\Acknowledgment}{OT1}{cmr}{bx}{n}{14pt}
\begin{document}

\title{\bf On the deformation of discrete conformal factors on surfaces}
\author{Huabin Ge, Wenshuai Jiang}
\maketitle

\begin{abstract}
In \cite{Luo0}, Feng Luo conjectured that the discrete Yamabe flow will converge to the constant curvature PL-metric after finite number of surgeries on the triangulation. In this paper, we prove that the flow can always be extended (without surgeries) to a solution that converges exponentially fast to the constant curvature PL-metric.
\end{abstract}


\section{Introduction}\label{Introduction}
\subsection{Background}
It appears that conformal transformations in the discrete settings were first considered by Ro$\check{c}$ek and Williams in \cite{Rocek}. Feng Luo first introduced the notion of discrete conformal factor on triangulated surfaces \cite{Luo0}. In the same paper, Feng Luo proposed to find the PL-metric of constant curvature in the PL-conformal class by discrete Yamabe flow. Luo studied his discrete Yamabe flow by making it as the negative gradient flow of a potential functional $F$. Since the domain where discrete conformal factor makes sense is not convex, the functional $F$ is only locally convex. Hence there are only local rigidity results for curvature map and local convergence results for discrete Yamabe flow. Surprisingly, Bobenko, Pincall and Springborn \cite{Bobenko} expressed $F$ in a concrete, clear and plain form. They further proved that $F$ can be extended to a convex continuously differentiable function defined on the whole Euclidean space. Using Bobenko-Pincall-Springborn's idea, Luo \cite{Luo1} introduce a method to extend $1$-forms and locally convex functions systematically. With the help of extended convex functional $F$, Bobenko-Pincall-Springborn \cite{Bobenko} affirmed a conjecture of Feng Luo \cite{Luo0} that the constant curvature PL-metric is unique in its discrete conformal class. In \cite{Luo0}, Feng Luo also conjectured that the discrete Yamabe flow will converge to the constant curvature PL-metric after finite number of surgeries on the triangulation. In this paper, we approach Luo's question in a different way. We do not need to perform surgeries on the triangulation. If the constant curvature PL-metric exists, then we show that the solution to Luo's discrete Yamabe flow can be extended automatically so as the extended solution converges exponentially fast to a constant curvature PL-metric. On the contrary, if the extended solution converges, then the constant curvature PL-metric exists. The main theorems in the paper are proved using variational principles. In the proof of the main results, we first need to extend the potential functional $F$ which is based on the work of \cite{Bobenko,Luo1}. We also need to show that the extended functional $\widetilde{F}$ is proper, which is almost the same with the authors former work \cite{Ge-Jiang1}.

\subsection{A dictionary of discrete conformal transformation}
Suppose $M$ is a closed surface with a triangulation $\mathcal{T}=\{V,E,F\}$,
where $V,E,F$ represent the sets of vertices, edges and faces respectively.
Let $d: E\rightarrow (0,+\infty)$ be a positive function assigning each edge $\{ij\}$ a length $d_{ij}$. We call $d$ a PL-metric (piecewise linear metric) if
for all triangle $\{ijk\}\in F$, three edge lengthes $d_{ij}$, $d_{jk}$ and $d_{ik}$ satisfy triangle inequalities. The triple $(M, \mathcal{T}, d)$ will be referred to as a PL-surface in the following. We assuming all the vertices are ordered one by one, marked by $1, \cdots, N$, where $N=V^\sharp$
is the number of vertices, and throughout this paper, all functions $f: V\rightarrow \mathds{R}$ will be regarded as column
vectors in $\mathds{R}^N$ and $f_i$ is the value of $f$ at $i$. We use $\mathcal{M}_\mathcal{T}$ as the space of all PL-metrics on the triangulated surface $(M, \mathcal{T})$ throughout this paper. For a fixed PL-surface $(M, \mathcal{T}, d)$, where $d\in\mathcal{M}_\mathcal{T}$ is a PL-metric, each triangle $\{ijk\}\in F$ is in fact an Euclidean triangle with edge lengthes $d_{ij}$, $d_{jk}$ and $d_{ik}$ and $M$ could be taken as gluing many Euclidean triangles coherently. Suppose $\theta_i^{jk}$ is the inner angle of the triangle $\{ijk\}$ at the vertex $i$, the classical well-known discrete Gaussian curvature at each vertex $i$ is defined as
\begin{equation}\label{classical-Gauss-curv}
K_i=2\pi-\sum_{\{ijk\} \in F}\theta_i^{jk},
\end{equation}
and the discrete Gaussian curvature $K_i$ satisfies the following discrete version of Gauss-Bonnet formula \cite{CL1}:
\begin{equation}\label{Gauss-Bonnet}
\sum_{i=1}^NK_i=2\pi \chi(M).
\end{equation}
By the discrete Gauss-Bonnet formula (\ref{Gauss-Bonnet}), the average of total discrete Gaussian curvature
\begin{equation}
K_{av}=\frac{\sum_{i=1}^NK_i}{N}=\frac{2\pi \chi(M)}{N}.
\end{equation}
is determined only by the topological and combinatorial information of $M$.

For any function $u:V\rightarrow\mathds{R}$ defined on all vertices, or say $u\in\mathds{R}^N$, and for each PL-metric $d\in\mathcal{M}_\mathcal{T}$, there corresponds to a new function $u*d:E\rightarrow (0,+\infty)$, with edge length
\begin{equation}
(u*d)_{ij}=e^\frac{u_i}{2}e^\frac{u_j}{2}d_{ij}
\end{equation}
for each edge $\{ij\}\in E$. Note generally $(u*d)_{ij}$, $(u*d)_{jk}$ and $(u*d)_{ik}$ may not satisfy triangle inequality for certain triangle $\{ijk\}\in F$, hence $u*d$ may not be a PL-metric again. For a given PL-metric $d\in\mathcal{M}_\mathcal{T}$, we denote
\begin{equation}
\Omega_{d}=\left\{u\in\mathds{R}^N\big|u*d\in\mathcal{M}_\mathcal{T}\right\}.
\end{equation}
Each $u\in\Omega_{d}$ is called a discrete conformal factor, and $u*d$ is called a discrete conformal transformation of $d$. Moreover, the two PL-metrics $u*d$ and $d$ are called discretely conformal equivalent. This defines an equivalence relation on $\mathcal{M}_\mathcal{T}$. The discrete conformal class of a PL-metric $d$ is denoted as
\begin{equation}
[d]=\left\{u*d\big|u\in\Omega_{d}\right\}.
\end{equation}

\subsection{Discrete Yamabe flow}
Now we fix a PL-surface $(M, \mathcal{T}, d)$, where $d\in\mathcal{M}_\mathcal{T}$ is a fixed PL-metric and consider the discrete conformal transformation of $d$. We use $K(u)=K(u*d)$ as the discrete Gaussian curvature corresponding to the PL-metric $u*d$ for each $u\in\Omega_{d}$. Then we get a curvature map
$K:\Omega_{d}\rightarrow\mathds{R}^N, \;\;u\mapsto K(u)=K(u*d)$. Note here and in the following $K$ represents the curvature of metric $u*d$, not metric $d$. The discrete Yamabe problem asks if there is a constant curvature PL-metric within the discrete conformal class $[d]$. Feng Luo once conjectured that the constant curvature PL-metric is unique in its discrete conformal class, this is proved by Bobenko, Pincall and Springborn in \cite{Bobenko}. As to the existence of constant curvature PL-metrics, Feng Luo proposed to find the constant curvature PL-metric in the class $[d]$ by a system of ordinary differential equations
\begin{equation}\label{Def-Luo-normalized-Yamabe-flow}
\begin{cases}
{u_i}'(t)=K_{av}-K_i(u)\\
\;\,u(0)=0.
\end{cases}
\end{equation}
which is called (normalized) discrete Yamabe flow. Consider the deformation of discrete conformal factor $u$ along flow (\ref{Def-Luo-normalized-Yamabe-flow}). Note that, $K_i(u)$ as a function of $u$ is smooth and hence locally Lipschitz continuous. By Picard theorem in classical ODE theory, flow (\ref{Def-Luo-normalized-Yamabe-flow}) has a unique solution $u(t)$, $t\in[0, \epsilon)$ for some $\epsilon>0$. However, the flow  (\ref{Def-Luo-normalized-Yamabe-flow}) may develop removable singularities (that is, the triangle inequalities for some triangles may not satisfied any more) in finite time, and the flow stops when ever removable singularity occurs. If removable singularity occurs, say the vertex $k$ is moving toward the interior of the edge $\{ij\}$, and $\{ijk\}$ and $\{ijl\}$ are two triangles with a common edge $\{ij\}$, then one can replace the triangulation $\mathcal{T}$ by a new triangulation $\mathcal{T}'$ which is obtained from $\mathcal{T}$ by deleting the edge $\{ij\}$ and adding a new edge  $\{kl\}$. This ``edge flip" operation is called a surgery on the triangulation, and Feng Luo conjectured that the discrete Yamabe flow will converges to the constant curvature PL-metric after finite number of surgeries on the triangulation.

In this paper, we shall prove that the solution of flow (\ref{Def-Luo-normalized-Yamabe-flow}) can always be extended (without surgeries) to a solution that exists for all time $t\geq 0$. Moreover, the extended solution converges exponentially fast to the constant curvature PL-metric when it exists. The basic idea is to extend the definition of curvature $K$ continuously to a generalized curvature $\widetilde{K}$, which is defined on for all $u\in \mathds{R}^N$. Even the triangle inequalities are not satisfied, $\widetilde{K}$ is still well defined and uniformly bounded by topological and combinatorial data. As to flow (\ref{Def-Luo-normalized-Yamabe-flow}), even if the triangular inequalities may not valid in some finite time, we can still deform the discrete conformal factors along an extended flow until time tends to $+\infty$. We shall prove the main extension and convergence theorem in next section:

\begin{theorem}\label{Thm-main-1}
Given a PL-surface $(M, \mathcal{T}, d)$, where $d\in\mathcal{M}_\mathcal{T}$ is a fixed PL-metric. Suppose $\{u(t)|t\in[0,T)\}$ is the unique maximal solution to flow (\ref{Def-Luo-normalized-Yamabe-flow}) with $0<T\leq +\infty$. Then we can always extend it to a solution $\{u(t)|t\in[0,+\infty)\}$ when $T<+\infty$. This is equivalent to say, the solution to the extended flow
\begin{equation}\label{Def-extended-flow}
\begin{cases}
{u_i}'(t)=K_{av}-\widetilde{K}_i(u)\\
\;\,u(0)=0
\end{cases}
\end{equation}
exists for all time $t\in[0,+\infty)$. Furthermore, the extended solution converges if any only if there exists a PL-metric of constant curvature. More precisely, on one hand, if the extended solution $u(t)$ converges to some $u(\infty)\in\Omega_d$, then $u(\infty)$ is a PL-metric of constant curvature. On the other hand, assuming there exists a PL-metric $u^*\in\Omega_d$ of constant curvature, then $u(t)$ can always be extended to a solution that converges exponentially fast to the PL-metric of constant curvature as $t\rightarrow+\infty$.
\end{theorem}

\section{Proofs of Theorem \ref{Thm-main-1}}\label{section-prove}
\subsection{Extend the definition of curvature}
Set $\Delta=\left\{(x_1, x_2, x_3)\big|\;x_1+x_2>x_3>0,\;x_1+x_3>x_2>0,\;x_2+x_3>x_1>0\right\}$. It is obviously that $\Delta$ is a open convex cone of $\mathds{R}^3_{>0}$. For each $(x_1, x_2, x_3)\in\Delta$, there corresponds to an Euclidean triangle $\triangle 123$, with three edge lengths $x_1$, $x_2$ and $x_3$. Let $\theta_i$ be the inner angle in triangle $\triangle 123$ that faces edge $x_i$ for each $i\in\{1, 2, 3\}$. Thus we get a inner angle map $\theta: \Delta\rightarrow \mathds{R}^3_{>0}$, sending three edge lengths $x_1$ ,$x_2$ and $x_3$ to three inner angles $\theta_1$, $\theta_2$ and $\theta_3$.

\begin{definition}(\cite{Bobenko,Luo1,Luo2})
A generalized Euclidean triangle $\triangle$ is a (topological) triangle of vertices $v_1, v_2, v_3$ so that each edge is assigned a positive number, called edge
length. Let $x_i$ be the assigned length of the edge $v_jv_k$ where $\{i, j, k\}$=$\{1, 2, 3\}$. The inner angle $\tilde{\theta}_i$=$\tilde{\theta}_i(x_1, x_2, x_3)$ at the vertex $v_i$ is defined as follows. If $x_1, x_2, x_3$ satisfy the triangle inequalities that $x_j+x_k>x_h$ for $\{h, j, k\}$=$\{1, 2, 3\}$, then $\tilde{\theta}_i$ is the inner angle of the Euclidean triangle of edge lengths $x_1, x_2, x_3$ opposite to the edge of length $x_i$; if $x_i\geq x_j+x_k$, then $\tilde{\theta}_i=\pi$, $\tilde{\theta}_j=\tilde{\theta}_k=0$.
\end{definition}
Thus the angle map $\theta: \Delta\rightarrow \mathds{R}^3_{>0}$ is extended to $\tilde{\theta}: \mathds{R}^3_{>0}\rightarrow \mathds{R}^3_{\geq0}$. It is known that
\begin{lemma}(Luo \cite{Luo1})
The angle function $\tilde{\theta}: \mathds{R}^3_{>0}\rightarrow \mathds{R}^3_{>0}$, $(x_1,x_2,x_3)\mapsto(\tilde{\theta}_1,\tilde{\theta}_2,\tilde{\theta}_3)$ is continuous so that $\tilde{\theta}_1+\tilde{\theta}_2+\tilde{\theta}_3=\pi$.\qed
\end{lemma}

Now consider the fixed PL-surface $(M,\mathcal{T},d)$ and the discrete conformal transformation $u*d$ of a PL-metric $d$. For each $u\in\mathds{R}^N$, $(u*d)_{ij}$, $(u*d)_{jk}$ and $(u*d)_{ik}$ form the three edge lengthes of a generalized Euclidean triangle for each $\{ijk\}\in F$. Hence the angle map $\tilde{\theta}_i^{jk}=\tilde{\theta}_i^{jk}(u)$, $u\in\mathds{R}^N$ is well defined and is the continuous extension of function $\theta_i^{jk}=\theta_i^{jk}(u)$, $u\in\Omega_d$. Recall the definition of discrete Gaussian curvature $K_i=2\pi-\sum_{\{ijk\} \in F}\theta_i^{jk}$, which is originally defined on $\Omega_d$, now can naturally extends continuously to
\begin{equation}\label{def-K-tuta}
\widetilde{K}_i=2\pi-\sum_{\{ijk\} \in F}\tilde{\theta}_i^{jk}.
\end{equation}
Hence the curvature map $K(u):\Omega_d\rightarrow \mathds{R}^N$ is extended continuously to
$\widetilde{K}(u):\mathds{R}^N\rightarrow \mathds{R}^N$ with discrete Gauss-Bonnet formula remain valid.
\begin{proposition}(\cite{Ge-Jiang1}) For the extended curvature $\widetilde{K}$, the following discrete Gauss-Bonnet formula holds.
\begin{equation}\label{Gauss-Bonnet-extend}
\sum_{i=1}^N\widetilde{K}_i=2\pi \chi(M).
\end{equation}
\end{proposition}

\begin{lemma}(\cite{Luo0})
For any fixed PL-metric $d\in\mathcal{M}_\mathcal{T}$, $\Omega_d\subset\mathds{R}^N$ is open, simply connected and nonempty (in fact, $0\in\Omega_d$). Moreover, if $u\in\Omega_d$, then $u+t\mathds{1}\in\Omega_d$ for any $t\in\mathds{R}$.
\end{lemma}

By calculation, one can get $\frac{\partial K_i}{\partial u_j}=\frac{\partial K_j}{\partial u_i}$ on $\Omega_d$, hence $\sum_{i=1}^N(K_i-K_{av})du_i$ is a closed $C^{\infty}$-smooth $1$-form on $\Omega_d$. Note $\Omega_d$ is simply connected, hence the potential functional
\begin{equation}\label{def-functional-F}
F(u)\triangleq\int_{0}^u \sum_{i=1}^N(K_i-K_{av})du_i,\;\;u\in\Omega_d
\end{equation}
is well defined. This type of functional was first constructed by Colin de Verdi$\grave{e}$re \cite{Colindev}. Potential functional (\ref{def-functional-F}) is locally convex on
$\Omega_d$, in fact, we have
\begin{proposition}
On $\Omega_d$, $F(u)$ is $C^{\infty}$-smooth. $Hess_uF$ is positive semi-definite with rank $N-1$ and null space $\{t\mathds{1}|t\in\mathds{R}\}$.
\end{proposition}
\begin{proof}
The properties of $Hess_uF=\frac{\partial(K_{1},\cdots,K_{N})}{\partial(u_{1},\cdots,u_{N})}$ is essentially proved in \cite{Luo0, Glickenstein2}. For a single triangle $\{ijk\}\in F$, the matrix
$$L_{ijk}\triangleq-\frac{\partial(\theta_i^{jk}, \theta_j^{ik}, \theta_k^{ij})}{\partial(u_i, u_j, u_k)}$$
is positive semi-definite with rank $2$ and null space $\{t(1,1,1)^T|t\in\mathds{R}\}$. For the proof of this fact, see \cite{Glickenstein2,Luo0,Bobenko}. For each edge $\{ij\}\in E$ which is a common edge of triangle $\{ijk\}$ and triangle $\{ijl\}$, the $(i,j)$-entry of $Hess_uF$ is $\frac{\cot\theta_{i}^{jk}+\cot\theta_{i}^{jl}}{2}$. This shows that $Hess_uF$ is in fact the discrete Laplace matrix \cite{Chung} with cotangent edge weight. This type of discrete Laplace matrix is extensively studied especially in computational geometry and in finite element area. We omit the detailed proofs here.\qed\\
\end{proof}

To prove Theorem \ref{Thm-main-1}, we need to extend $F$ to a convex functional $\widetilde{F}$ defined on the whole space $\mathds{R}^N$. There are two ways to extend $F$, one way is Feng Luo's approach \cite{Luo1}, the other is Bobenko-Pincall-Springborn's approach \cite{Bobenko}, we state them separately.

\subsection{Extend potential $F$: Feng Luo's approach}
A differential $1$-form $\omega=\sum_{i=1}^na_i(x)dx_i$ in an open set $U\subset\mathds{R}^n$ is said to be continuous if each $a_i(x)$ is a continuous function on $U$. A continuous $1$-form $\omega$ is called closed if $\int_{\partial\tau}\omega=0$ for any Euclidean triangle $\tau\subset U$. By the standard approximation theory, if $\omega$ is closed and $\gamma$ is a piecewise $C^1$-smooth null homologous loop in $U$, then $\int_{\gamma}\omega=0$. If $U$ is simply connected, then in the integral
$G(x)=\int_{a}^x\omega$ is well defined (where $a\in U$ is arbitrary chose), independent of the choice of piecewise smooth paths in $U$ from $a$ to $x$. Moreover, the function $G(x)$ is $C^1$-smooth so that $\frac{\partial G(x)}{\partial x_i}=a_i(x)$.

\begin{lemma}\label{lemma-luo-essential}(Luo \cite{Luo1})
Suppose $X\subset\mathds{R}^n$ is an open convex set and $A\subset X$ is an open and simply connected subset of $X$ bounded by a real analytic codimension-1 submanifold in $X$. If $\omega=\sum_{i=1}^na_i(x)dx_i$ is a continuous closed $1$-form on $A$ so that $F(x)=\int_{a}^x\omega$ is locally convex on $A$ and each $a_i$ can be extended continuously to $X$ by constant functions to a function $\widetilde{a}_i$ on $X$, then $\widetilde{F}(x)=\int_{a}^x\widetilde{a}_i(x)dx_i$ is a $C^1$-smooth convex function on $X$ extending $F$.\qed
\end{lemma}

Lemma \ref{lemma-luo-essential} plays an essential role in \cite{Luo1} and \cite{Luo2}. We use this lemma to extend the definition of $F(u)$. Consider the whole triangulation $(M,\mathcal{T},d)$ with PL-metric $d\in\mathcal{M}_\mathcal{T}$. Recall that $\tilde{\theta}_i=\tilde{\theta}_i(u)$, $u\in\mathds{R}^N$ is the continuous extension of function $\theta_i=\theta_i(u)$, $u\in \Omega_d$. It's easy to see, $\tilde{\theta}_i du_i+\tilde{\theta}_j du_j+\tilde{\theta}_kdu_k$ is a continuous closed $1$-form  on $\mathds{R}^N$, hence the following integration
\begin{equation}
\widetilde{F}_{ijk}(u)\triangleq\int_{0}^{u}\tilde{\theta}_i du_i+\tilde{\theta}_j du_j+\tilde{\theta}_k du_k, \,\,\,u\in\mathds{R}^N
\end{equation}
is well defined and is a $C^1$-smooth concave function on $\mathds{R}^N$. Note that
\begin{equation*}
\begin{aligned}
\sum_{i=1}^N(\widetilde{K}_i-K_{av})du_i=&\sum_{i=1}^N\Big(2\pi-K_{av}-\sum_{\{ijk\}\in F}\tilde{\theta}_i^{jk}\Big)du_i\\
=&\sum_{i=1}^N(2\pi-K_{av})du_i-\sum_{i=1}^N\sum_{\{ijk\}\in F}\tilde{\theta}_i^{jk}du_i\\
=&\sum_{i=1}^N(2\pi-K_{av})du_i-\sum_{\{ijk\}\in F}\Big(\tilde{\theta}_i^{jk}du_i+\tilde{\theta}_j^{ik}du_j+\tilde{\theta}_k^{ij}du_k\Big),
\end{aligned}
\end{equation*}
which shows that $\sum_{i=1}^N(\widetilde{K}_i-K_{av})du_i$ is a continuous closed $1$-form on $\mathds{R}^N$. Therefore, the extended potential
\begin{equation}
\widetilde{F}(u)\triangleq\int_{u_0}^u \sum_{i=1}^N(\widetilde{K}_i-K_{av})du_i, \,\,\, u\in \mathds{R}^N
\end{equation}
is well defined and $C^1$-smooth that extends potential $F$.

\subsection{Extend potential $F$: Bobenko-Pincall-Springborn's approach}
Bobenko-Pincall-Springborn's approach is very surprising, they actually calculated the concrete expression of potential $F$. Set
\begin{equation}
L(x)=-\int_0^x\log|2\sin(t)|dt,
\end{equation}
it is called Milnor's Lobachevsky function \cite{Bobenko}, and consider the function
\begin{equation}\label{def-f-bobenko-springborn-pinkall}
f(x,y,z)=\alpha x+\beta y+\gamma z+L(\alpha)+L(\beta)+L(\gamma),
\end{equation}
where $\alpha$, $\beta$, and $\gamma$ are the angles in a Euclidean triangle that face sides $e^x$, $e^y$, and $e^z$ respectively. $f$ is (for now) only defined on the set $\{(x,y,z)|(e^x,e^y,e^z)\in\Delta\}$ and is locally convex. It's Bobenko-Pincall-Springborn's pioneered observation that the definition of $f$ can be extend to $\mathds{R}^3$ as follows. Define $f(x,y,z)$ by equation (\ref{def-f-bobenko-springborn-pinkall}) for all $(x,y,z)\in\mathds{R}^3$, where for $(x,y,z)$ do not contained in the set $\{(x,y,z)|(e^x,e^y,e^z)\in\Delta\}$, the angles $\alpha$, $\beta$, and $\gamma$ are defined to be $\pi$ for the angle opposite the side that is too long and $0$ for the other two. The so extended function $f:\mathds{R}^3\rightarrow\mathds{R}$ is $C^1$-continuously differentiable and convex.

For each edge $\{ij\}\in E$, set $\lambda_{ij}=2\ln (u*d)_{ij}$, then $\lambda_{ij}=u_i+u_j+2\ln d_{ij}$. Following Bobenko-Pincall-Springborn's approach, we consider the function
\begin{equation}
\mathcal{E}(u)=\sum_{\{ijk\}\in F}\Big(2f\Big(\frac{\lambda_{ij}}{2},\frac{\lambda_{jk}}{2},\frac{\lambda_{ik}}{2}\Big)-
\frac{\pi}{2}\big(\lambda_{ij}+\lambda_{jk}+\lambda_{ik}\big)\Big)+\sum_{i=1}^N\big(2\pi-K_{av}\big)u_i,\;u\in\mathds{R}^N.
\end{equation}
Obviously, $\mathcal{E}$ is convex and $C^1$-smooth on $\mathds{R}^N$. Moreover, it's easy to get
\begin{equation}
\nabla_u\mathcal{E}=\nabla_u\widetilde{F}=(\widetilde{K}-K_{av}\mathds{1})^T
\end{equation}
on $\mathds{R}^N$. Specifically, $\nabla_u\mathcal{E}=\nabla_uF=(K-K_{av}\mathds{1})^T$ on $\Omega_d$. Hence $\mathcal{E}(u)$ differs from $\widetilde{F}(u)$ only by a constant on $\mathds{R}^N$, that is,
\begin{equation}
\mathcal{E}(u)=\widetilde{F}(u)+\mathcal{E}(0).
\end{equation}
Thus $\mathcal{E}(u)-\mathcal{E}(0)$ is exactly the extension of $F(u)$ from $u\in\Omega_d$ to $u\in\mathds{R}^N$.

\subsection{Properness of potential $F$}
\begin{theorem}\label{thm-F-tura-proper-tends-infinity}
Given a PL-surface $(M,\mathcal{T},d)$. Assuming there exists a discrete conformal factor $u_{av}\in\Omega_d$ so that $u_{av}*d$ is a PL-metric of constant curvature. Set $\mathscr{U}\triangleq \{u\in \mathds{R}^N|u^T\mathds{1}=0\}$. Then $\widetilde{F}(u)$ is proper on $\mathscr{U}$ and
$\lim\limits_{u\in \mathscr{U}, u\rightarrow\infty}\widetilde{F}(u)=+\infty.$
\end{theorem}
\begin{proof}
The proof appeared in Proposition 3.9 of \cite{Ge-Jiang1} can be used here word for word, only with $\ln \Omega$ changed here by $\Omega_d$. However, for completeness of this paper and for user's convenience, we still state the detailed proof here. Obviously, $\widetilde{F}(u)\in C^1(\mathds{R}^N)$, and $\nabla_u\widetilde{F}=\widetilde{K}-K_{av}\mathds{1}$. For each direction $\xi\in \mathbb{S}^{N-1}\cap\mathscr{U}$, set $\varphi_{\xi}(t)=\widetilde{F}(u_{av}+t\xi)$, $t\in \mathds{R}$. Obviously $\varphi_{\xi}\in C^1(\mathds{R})$, and $\varphi_{\xi}'(t)=(\widetilde{K}-K_{av}\mathds{1})\cdot\xi$. $\varphi_{\xi}(t)$ is convex on $\mathds{R}$ since $\widetilde{F}$ is convex on $\mathds{R}^N$, hence $\varphi'_{\xi}(t)$ is increasing on $\mathds{R}$. Note $u_{av}\in \Omega_d$, hence there exists $c>0$ so that for each $t\in[-c, c]$, $u_{av}+t\xi$ remains in $\Omega_d$. Note $\widetilde{F}(u)$ is $C^{\infty}$-smooth on $\Omega_d$, hence $\varphi_{\xi}(t)$ is $C^{\infty}$-smooth for $t\in[-c, c]$.
Note that the kernel space of $Hess_u\widetilde{F}$ is $\{t\mathds{1}|t\in\mathds{R}\}$, which is perpendicular to the hyperplane $\mathscr{U}$. Hence $\widetilde{F}(u)$ is strictly convex locally in $\Omega_d\cap\mathscr{U}$. Then $\varphi_{\xi}(t)$ is strictly convex at least on a small interval $[-\delta, \delta]$. This implies that $\varphi'_{\xi}(t)$ is a strictly increase function on $[-\delta, \delta]$. Note that $\varphi'_{\xi}(0)=0$, hence $\varphi'_{\xi}(t)\geq\varphi'_{\xi}(\delta)>0$ for $t>\delta$ while $\varphi'_{\xi}(t)\leq\varphi'_{\xi}(-\delta)<0$ for $t<-\delta$. Hence $\varphi_{\xi}(t)\geq\varphi_{\xi}(\delta)+\varphi'_{\xi}(\delta)(t-\delta)$ for $t\geq \delta$, while
$\varphi_{\xi}(t)\geq\varphi_{\xi}(-\delta)+\varphi'_{\xi}(-\delta)(t+\delta)$ for $t\leq -\delta$. This implies $\lim\limits_{t\rightarrow\pm\infty}\varphi_{\xi}(t)=+\infty$. Next by the following Lemma \ref{lemma-gotoinfinity}, we get the conclusion.\qed
\end{proof}

\begin{lemma}\label{lemma-gotoinfinity}(\cite{Ge-Jiang1})
Assuming $f\in C(\mathds{R}^n)$ and for any direction $\xi\in \mathbb{S}^{n-1}$, $f(t\xi)$ as a function of $t$ is monotone increasing on $[0,+\infty)$ and tends to $+\infty$ as $t\rightarrow +\infty$. Then $\lim\limits_{x\rightarrow\infty}f(x)=+\infty.$
\end{lemma}
\begin{proof}
The proof is element, see Lemma 3.10 in \cite{Ge-Jiang1}.\qed\\
\end{proof}

In \cite{Bobenko}, Bobenko, Pincall and Springborn posed discrete conformal mapping problem and proved it has at most one solution up to scale (see \cite{Bobenko}, Corollary 3.1.5). In our language here, it is equivalent to say, the curvature map $K(u):\Omega_d\rightarrow \mathds{R}^N$ is injective up to scale, or say, $K(u):\Omega_d\cap \mathscr{U}\rightarrow \mathds{R}^N$ is injective. If $K(u)=K(u')$ for two $u, u'\in \Omega_d$, then $u'=u+t\mathds{1}$ for some $t\in\mathds{R}$. This conforms a conjecture of Feng Luo in \cite{Luo0} that the constant curvature PL-metric is unique in its discrete conformal class. However, we can prove more.
\begin{theorem}\label{corollary-unique-K-average}
Given a PL-surface $(M,\mathcal{T},d)$. Assuming there exists a discrete conformal factor $u_{av}\in\Omega_d$ so that $u_{av}*d$ is a PL-metric of constant curvature. Then for the extended curvature $\widetilde{K}$ defined on the extended space $\mathds{R}^N$, $u_{av}$ is the unique value (up to scale) so that the corresponding extended curvature is constant.\qed
\end{theorem}
\begin{proof}
Suppose $\hat{u}\in \mathscr{U}$ is also a discrete conformal factor such that $\hat{u}*d$ is a metric of constant curvature, while $\hat{u}\neq u_{av}$. Write $t_0=\|\hat{u}-u_{av}\|$ and  $\xi=t_0^{-1}(\hat{u}-u_{av})$. Then $\varphi_{\xi}'(t_0)=(\widetilde{K}(\hat{u})-K_{av}\mathds{1})\cdot\xi=0$. On the other hand, we had proved in Proposition \ref{thm-F-tura-proper-tends-infinity} that $t=0$ is the unique zero point of $\varphi_{\xi}'(t)$, which is a contradiction. \qed
\end{proof}

\begin{theorem}\label{thm-section3-mainthm}
The solution to the extended flow
\begin{equation}
\begin{cases}
{u_i}'(t)=K_{av}-\widetilde{K}_i\\
\;\,u(0)\in \mathds{R}^N
\end{cases}
\end{equation}
exists for all time $t\in[0,+\infty)$. If further assuming there exists a discrete conformal factor $u_{av}\in \Omega_d$ so that $u_{av}*d$ is a constant curvature PL-metric, then the solution converges to $u_{av}$ exponentially fast.\qed
\end{theorem}

Theorem \ref{thm-section3-mainthm} implies Theorem \ref{Thm-main-1}. Thus we finally proved Theorem \ref{Thm-main-1}. However, these theorems are still true if we replace the constant curvature to any admissible prescribed curvature. In the following, we use a prescribed discrete Yamabe flow to deform any PL-metric to a metric with admissible prescribed curvature.

\begin{definition}
Given a PL-surface $(M,\mathcal{T},d)$. $\Omega_d$ is the space of all possible discrete conformal factors, while $[d]$ is the discrete conformal class determined by $d$. Consider $K$ as the function of $u$ and denote $K(\Omega_d)\triangleq\big\{K(u)\big|u\in \Omega_d\big\}$. Each prescribed $\bar{K}$ with $\bar{K}\in K(\Omega_d)$ is called admissible. If $\bar{u}\in\Omega_d$ such that $\bar{K}=K(\bar{u})$, we say $\bar{K}$ is realized by $\bar{u}$ (note that the $\bar{u}$ is unique (up to a scalar addition) in $\Omega_d$ that realizes $\bar{K}$ by Bobenko-Pincall-Springborn \cite{Bobenko}).
\end{definition}

\begin{theorem}\label{corollary-unique-K-bar}
Assuming $\bar{K}\in K(\Omega_d)$ is admissible. Then it is realized by an unique discrete conformal factor $\bar{u}$ in the extended space $\mathds{R}^N$ up to a scalar addition.\qed
\end{theorem}
Moreover, using the following prescribed Yamabe flow, we can deform any (generalized) discrete conformal factor $u\in\mathds{R}^N$ to any discrete conformal factor with admissible prescribed curvature.
\begin{theorem}\label{thm-section4-mainthm}
Consider the prescribed Yamabe flow
\begin{equation}\label{def-extended-flow-u-prescribed}
\begin{cases}
{u_i}'(t)=\bar{K}_{i}-\widetilde{K}_i\\
\;\,u(0)\in\mathds{R}^N
\end{cases}
\end{equation}
\begin{description}
  \item[(1)] The solution $u(t)$ to flow (\ref{def-extended-flow-u-prescribed}) exists for all time $t\in[0,+\infty)$.
  \item[(2)] If $u(t)$ converges to some $\bar{u}$, then $\bar{K}_{i}$ is admissible and is realized by $\bar{u}$.
  \item[(3)] If $\bar{K}_{i}$ is admissible and is realized by $\bar{u}$, then $u(t)$ converges to $\bar{u}$ exponentially fast.
\end{description}\qed
\end{theorem}

\begin{corollary}
Given a PL-surface $(M,\mathcal{T},d)$. Consider the following prescribed Yamabe flow
\begin{equation}
\begin{cases}
{u_i}'(t)=\bar{K}_{i}-K_i\\
\;\,u(0)=0.
\end{cases}
\end{equation}
Suppose $\{u(t)|t\in[0,T)\}$ is the unique maximal solution with $0<T\leq +\infty$. Then we can always extend it to a solution $\{u(t)|t\in[0,+\infty)\}$ when $T<+\infty$. Furthermore, if $\bar{K}$ is admissible and is realized by $\bar{u}$, then $u(t)$ can always be extended to a solution that converges exponentially fast to $\bar{u}$ as $t\rightarrow+\infty$.\qed
\end{corollary}
 
To prove these results, one just need to study the prescribed discrete functional  
\begin{equation}
G(u)\triangleq\int_{u'_0}^u \sum_{i=1}^N(K_i-\bar{K}_{i})du_i+C, \,\,\, u\in\Omega_d,
\end{equation}
where $u'_0\in \Omega_d$ is arbitrary chosen, and its $C^1$-smooth extension
\begin{equation}\label{def-G-tuta-conformal}
\widetilde{G}(u)\triangleq\int_{u_0}^u \sum_{i=1}^N(\widetilde{K}_i-\bar{K}_{i})du_i, \,\,\, u\in \mathds{R}^N.
\end{equation}
Since all the proofs in this subsection are similar with the authors former paper \cite{Ge-Jiang1} and with last section, we omit the details here.

\section{Some questions}
Different from the smooth surface setting, there is a combinatorial obstruction for the existence of the constant curvature PL-metric associated to a triangulation $\mathcal{T}$. In fact the following is proved by Feng Luo \cite{Luo0} (For a finite set $X$, we
use $|X|$ to denote the number of elements in $X$.).
\begin{theorem}
Fix a triangulation $\mathcal{T}$ of a closed topological surface $M$. There exists a PL-metric of constant curvature associated to $\mathcal{T}$ if and only if for any proper subset $I$ of the vertices $V$ of $\mathcal{T}$,
\begin{equation}\label{condition-Luofeng}
\frac{|F_I|}{|I|}>\frac{|F|}{|V|}
\end{equation}
where $F_I$ is the set of all triangles having a vertex in $I$.
\end{theorem}

Although the constant curvature PL-metric is unique in its discrete conformal class, it is not unique in the space of all PL-metrics $\mathcal{M}_\mathcal{T}$. In fact, consider the tetrahedron triangulation of sphere $\mathbb{S}^2$, let the edges that faces each other have same length, meanwhile the three pairs of facing each other edges have length $a$, $b$ and $c$ respectively that satisfying triangle inequalities. Then this PL-metric has constant curvature. On the other hand, even though the combinatorial condition (\ref{condition-Luofeng}) is valid, we still now know
whether there is a PL-metric of constant curvature for some fixed discrete conformal class $[d]$. For given PL-surface $(M, \mathcal{T}, d)$, Theorem \ref{Thm-main-1} shows that the extended solution to discrete Yamabe flow (\ref{Def-extended-flow}) converges if and only if the constant curvature PL-metric exists in $[d]$, this gives a sufficient and necessary condition for the existence of constant curvature PL-metric, however, the analytic condition (flow (\ref{Def-extended-flow}) converges) is not easy to verify. We want to know if there are topological and combinatorial conditions for the existence of constant curvature PL-metric in a fixed discrete conformal class $[d]$. Furthermore, what is the exact range of curvature map $K$? We need to determine the shape of $K(\Omega_d)=K([d])$.\\

\noindent \textbf{Acknowledgements:} Both authors would like to thank Professor Gang Tian for constant encouragement. The research is supported by National Natural Science Foundation of China under grant no.11501027. The first author would also like to thank Professor Feng Luo for many helpful conversations.

Huabin Ge: hbge@bjtu.edu.cn

Department of Mathematics, Beijing Jiaotong University, Beijing 100044, P.R. China\\

Wenshuai Jiang: jiangwenshuai@pku.edu.cn

BICMR and SMS of Peking University, Yiheyuan Road 5, Beijing 100871, P.R. China


\begin{thebibliography}{50}
\setlength{\itemsep}{-2pt} \small
\bibitem{Bobenko} A. Bobenko, U. Pinkall, B. Springborn, \emph{Discrete conformal maps and ideal hyperbolic polyhedra}, Geom. Topol., 19 (2015), 2155-2215.

\bibitem{CL1} B. Chow, F. Luo, \emph{Combinatorial Ricci flows on surfaces}, J. Differential Geometry, 63 (2003), 97-129.

\bibitem{Chung} F. R. K. Chung, \emph{Spectral graph theory}, CBMS Regional Conference Series in Mathematics, 92. American Mathematical Society, Providence, RI, 1997.

\bibitem{Colindev} Y. C. de Verdi\`{e}re, \emph{Un principe variationnel pour les empilements de cercles}, Invent. Math., 104(3) (1991), 655-669.

\bibitem{Ge-Jiang1} H. Ge, W. Jiang, \emph{On the deformation of inversive distance circle packings}, in preparation.

\bibitem{Glickenstein1} D. Glickenstein, \emph{A combinatorial Yamabe flow in three dimensions}, Topology, 44(4) (2005), 791-808.

\bibitem{Glickenstein2} D. Glickenstein, \emph{Discrete conformal variations and scalar curvature on piecewise flat two- and three-dimensional manifolds}, J. Differential Geom., 87(2) (2011), 201-237.

\bibitem{Luo0} F. Luo, \emph{Combinatorial Yamabe flow on surfaces}, Commun. Contemp. Math., 6(5) (2004), 765-780.

\bibitem{Luo1} F. Luo, \emph{Rigidity of polyhedral surfaces, III}, Geom. Topol., 15 (2011), 2299-2319.

\bibitem{Luo2} F. Luo, T. Yang, \emph{Volume and rigidity of hyperbolic polyhedral 3-manifolds}, arXiv:1404.5365v2 [math.GT].

\bibitem{Ri} I. Rivin, \emph{Euclidean structures on simplicial surfaces and hyperbolic volume}, Ann. of Math., 139 (1994), 553-580.

\bibitem{Rocek} M. Ro$\check{c}$ek, R. M. Williams, \emph{The quantization of Regge calculus}, Z. Phys., C(21) (1984), 371-381.\\

\end{thebibliography}
\end{document}